\documentclass[12pt]{extarticle}
\usepackage{microtype,amssymb,amsmath,amsthm} 
\usepackage[total={148mm,207mm}]{geometry}
\usepackage[bitstream-charter]{mathdesign}
\usepackage[english]{babel}
\usepackage[utf8x]{inputenc}
\usepackage[colorinlistoftodos]{todonotes}
\usepackage[colorlinks=true, allcolors=blue]{hyperref}

\author{Johan Andersson\textsuperscript{*}\thanks{\textsuperscript{*}Email:johan.andersson@oru.se \, Address:Department of Mathematics, School of Science and Technology, {\"O}rebro University, {\"O}rebro, SE-701 82 Sweden. }}

\theoremstyle{plain}
\newtheorem{thm}{Theorem}

\newtheorem{lem}{Lemma}

\theoremstyle{definition}

\newtheorem{ack}{Acknowledgements}

\nonstopmode
\date{}
\def\cprime{$'$}

\newcommand{\R}{{\mathbb R}} 
\newcommand{\abs}[1]{{\left| {#1} \right|}} \newcommand{\p}[1]{{\left(
      {#1} \right)}}

\renewcommand{\Re}{\operatorname{Re}}

 \newcommand{\Oh}[1]{{O \p{#1}}}

\begin{document}

\title{On the growth of the $L^p$ norm of the Riemann zeta-function on the line $\Re(s)=1$}

\maketitle

\begin{abstract}
We prove that if $\delta>0$ and $p$ is real then
$$\sup_T \int_T^{T+\delta} \abs{\zeta(1+it)}^p dt <\infty,$$
if and only if $-1<p<1$. Furthermore,  we show the omega estimates
\begin{gather*}
  \int_T^{T+\delta} \abs{\zeta(1+it)}^{\pm 1} dt = \Omega(\log \log \log T),  \\  \int_T^{T+\delta}  \abs{\zeta(1+it)}^{\pm p} dt = \Omega((\log \log T)^{p-1}), \qquad (p>1)
\end{gather*}
which with the exception of an additional $\log \log \log T$ factor in the second estimate coincides with conditional (under the Riemann hypothesis) order estimates. 
We also prove weaker unconditional order estimates.
\end{abstract}

\tableofcontents

\section{Classical order and omega estimates}
The study of the Riemann zeta-function  on the line $\Re(s)=1$ has been studied by a lot of authors, starting with the work of  Hadamard \cite{Hadamard} and de la Vall{\'e}e-Poussin \cite{Poussin} who proved that $\zeta(1+it) \neq 0$ which implies the prime number theorem.  Assuming the Riemann hypothesis,  Littlewood \cite{Littlewood} showed that
\begin{gather} \label{A1}
 \zeta(1+it) \ll \log \log t, \qquad  \zeta(1+it)^{-1} \ll \log \log t. \\ \intertext{Bohr and Landau \cite{BohrLandau1,BohrLandau2,BohrLandau3} proved the corresponding omega-estimates} \label{A2}
 \zeta(1+it)=\Omega(\log \log t), \qquad  \zeta(1+it)^{-1}=\Omega(\log \log t),
\end{gather}
unconditionally, so Littlewood's conditional bound is the best possible. The best unconditional bound are the estimates
\begin{gather} \label{A3}
 \zeta(1+it) \ll (\log t)^{2/3}, \qquad  \zeta(1+it)^{-1} \ll ( \log t)^{2/3} (\log \log t)^{1/3},
\end{gather}
of Vinogradov \cite{Vinogradov} and Korobov \cite{Korobov}. For a discussion of these results as well as the current record, see the recent paper  of Granville-Soundararajan \cite{Sound}.

\section{The $L^p$ norm in short intervals}
\subsection{Bounds from below}
A related question which has been less studied is the question of the $p$'th moment of the Riemann zeta-function in short intervals. What can we say about
\begin{gather} \label{star}
 \int_T^{T+\delta} \abs{\zeta(1+it)}^p dt?
\end{gather}
One of our recent results \cite[Theorem 7]{Andersson1} is the following.
\begin{thm}   We have the following estimates for the $L^p$ norm, for $p>0$ of the zeta-function and its inverse in short intervals: 
   \begin{align*}
     (i)& \qquad   \inf_{T}  \p{\frac 1 {\delta} \int_T^{T+\delta} \abs{\zeta(1+it)}^p dt}^{1/p} = \frac{\pi^2 e^{-\gamma}}{24} \delta+\Oh{\delta^3}, \\
     (ii)& \qquad  \inf_{T}   \p{\frac 1 \delta \int_T^{T+\delta} \abs{\zeta(1+it)}^{-p} dt}^{1/p} = \frac{e^{-\gamma}} 4 \delta+\Oh{\delta^3}, 
\end{align*}
 for $\delta>0$. Furthermore, both estimates are valid if $\displaystyle \inf_T$ is replaced by $\displaystyle \liminf_{T \to \infty}$, and if $1+it$ is replaced by $\sigma+it$ and the infimum is also taken over $\sigma>1$.
\end{thm}
This result gives lower estimates for this integral. In particular it shows that the infimum  is strictly positive and thus gives an analogue of Hadamard and de la Vall\'ee Poussin's result for the non vanishing of the Riemann zeta-function on the line $\Re(s)=1$.

As discussed in \cite{Andersson1} this can be applied to the question of universality on the 1-line. It should be noted that in a surprise turn of events \cite{Andersson2} we recently managed to prove that a Voronin universality type result in fact do hold on the line Re$(s)=1$ if we in addition to vertical shift allow scaling in the argument and adding a positive constant in the range.

\subsection{Bounds from above}
Another question is whether we similarly as Littlewood's and Bohr's results can obtain omega, and order results for the quantity in \eqref{star}? This is the topic for the current paper.  The first question regarding this is whether for some $p>0$  this is bounded. This is answered in the following theorem
\begin{thm} \label{thm2}
  We have for real $p$ that 
   \begin{gather}
      \sup_T \int_T^{T+\delta} \abs{\zeta(1+it)}^p dt<\infty 
    \end{gather}
    if and only if $-1<p<1$. Furthermore $\displaystyle \sup_T$ can be replaced by $\displaystyle \lim \sup_T$ and the same result holds.
\end{thm} 
For $-1<p<1$ this implies similar results on the non universality of the Riemann zeta-function on the 1-line as Theorem 1 as we proved in \cite{Andersson1}. More precisely Theorem \ref{thm2} for $0<p<1$ gives an upper bound $M$ for the $L^p$-norm of the functions $\zeta_T(t):=\zeta(1+iT+it)$ on the interval $[0,\delta]$. 
 Thus the zeta-function can not approximate any  function $f$ with $L^p$ norm strictly greater than $M$. Thus this gives another proof of the fact that the usual Voronin universality theorem does not extend to the line $\Re(s)=1$.

It is reasonable to expect that $T=-\delta/2$ should maximize this integral for $p>0$. It is clear that this would imply Theorem 1 for positive values of $p$, since the integral with this value of $T$ is divergent exactly when 
  $p \geq 1$. We will not prove this, but rather leave it as an open problem. However we will manage to prove the corresponding result for the following related integral
  \begin{gather}
      \sup_T \int \abs{\zeta(1+it)}^p \theta\p{\frac {t-T} \delta} dt 
    \end{gather}
whenever the Fourier transform 
 \begin{gather} \label{ft} \hat \theta(\tau)= \int_{-\infty}^\infty e^{- 2 \pi i \tau x} \theta(x)dx \end{gather}
is non negative. For the purpose of this paper we will choose the triangular function
\begin{gather} \label{thetaref}
   \theta(x)=\begin{cases} 1-|x|, & |x| \leq 1,\\ 0, & |x|>1.\end{cases}
\end{gather}
For this integral kernal it is well known that its Fourier transform
\begin{gather} \label{ft2}
 \hat \theta(\tau)= \frac{(\sin  \pi \tau)^2}{\pi^2 \tau ^2}
\end{gather}
is non negative\footnote{This is essentially the Fourier transform of the Fej\'er kernel}.

\subsubsection{Integral kernals with non negative Fourier transforms and  Dirichlet series}
Before starting to prove Theorem 1 and Theorem 2 we prove some more general lemmas
\begin{lem}  \label{lem1}
 Let  
 \begin{gather*}
   L(s)=\sum_{n=1}^\infty a(n) n^{-s} 
 \end{gather*}
 be a Dirichlet series absolutely convergent on $\Re(s)=\sigma$, where  $a(n)=|a(n)| b(n)$ and where $b(n)$ is a completely multiplicative arithmetical function.   
 Then
  \begin{gather*}
   \limsup_{T \to \infty} \int_{T-\delta}^{T+\delta}  \abs{L(\sigma+it)}^2 \, (\delta- |t|) dt =   \int_{-\delta}^{\delta} \abs{\tilde L(\sigma+it)}^2  \, (\delta- |t|) dt, \\ \intertext{where} 
   \notag  \tilde L(s)=\sum_{n=1}^\infty \abs{a(n)} n^{-s},  
  \end{gather*}
and where $\limsup_{T \to \infty}{}$ may be replaced by $\sup_T{}$.  \end{lem}
\begin{proof} By using \eqref{ft}, \eqref{thetaref}, the fact that $\hat \theta(x)\geq 0$ is non negative and the triangle inequality we see that
\begin{align} 
  \int_{T-\delta}^{T+\delta}  \abs{L(\sigma+it)}^2 &\,  (\delta-|t|) dt \notag \\ 
  &= \sum_{m,n=1}^\infty \frac{a(n)\overline{a(m)}}{(nm)^\sigma} \int_{T-\delta}^{T+\delta}  \p{\frac n m}^{-it}  \delta \theta \p{\frac  t \delta}   dt, \notag  \\
   \notag
                            &=\sum_{m,n=1}^\infty \frac{a(n)\overline{a(m)}}{(nm)^\sigma}  \p{\frac n m}^{iT}  \int_{-\delta}^{\delta}  e^{-i \log \frac n m t}  \delta \theta \p{\frac  t \delta} dt, \notag \\                                             
                                        &=   \delta^2 \sum_{m,n=1}^\infty \frac{a(n)\overline{a(m)}}{(nm)^\sigma}   \p{\frac n m}^{iT}   \hat \theta\p{\frac \delta {2 \pi} \log \frac n m}, \label{uiii2}
                                         \\
                                          &\leq \delta^2 \sum_{m,n=1}^\infty \frac{\abs{a(n)} \abs{a(m)}}{(nm)^\sigma}     \hat \theta\p{\frac \delta {2 \pi} \log \frac n m},  \notag \\  
                                          &=\int_{-\delta}^{\delta} \abs{\tilde L(\sigma+it)}^2  \, (\delta- |t|) dt.   \notag 
  \end{align}
  By Kroneckers theorem we may choose $T$ such that 
  \begin{gather} \abs{b(P)-P^{-iT}}<\varepsilon, \qquad (P \text{ prime}, \, P<N_0), 
    \end{gather}
  and by choosing $\varepsilon$ sufficiently small and $N_0$ sufficiently large it follows by the fact that $L(s)$ is absolutely convergent on the line $\Re(s)=\sigma$ that  \eqref{uiii2} may be as close to
 \begin{gather*}    \delta^2 \sum_{m,n=1}^\infty \frac{ \abs{a(n)} \abs{a(m)}}{(nm)^\sigma}   \hat \theta\p{\frac{\delta}{2 \pi}\log \frac n m}=\int_{-\delta}^{\delta} \abs{\tilde L(\sigma+it)}^2  \, (\delta- |t|) dt
\end{gather*}
as we wish.
\end{proof}

From this Lemma we obtain the following result for the Riemann zeta-function.

\begin{lem} \label{lem2} Let $\sigma>1$ and $p \geq 0$. Then
 \begin{align*}
  (i)&  &   \sup_T \int_{T-\delta}^{T+\delta}  \abs{\zeta(\sigma+it)}^p \,  (\delta-|t|) dt&=  \int_{-\delta}^{\delta} \abs{\zeta(\sigma+it)}^p  \, (\delta-|t|) dt. \\
  (ii)&  &  \sup_T \int_{T-\delta}^{T+\delta}  \abs{\frac{\zeta(2 \sigma+2it)}{\zeta(\sigma+it)}}^p \,  (\delta-|t|) dt &=  \int_{-\delta}^{\delta} \abs{\zeta(\sigma+it)}^{p} \,  (\delta-|t|) dt. 
 \end{align*} 
 Furthermore $\sup_T$ may be replaced by $\lim\sup_{T}$.
\end{lem}

\begin{proof} In general we have the following equality for $\Re(s)>1$
\begin{gather}
 \zeta(s)^p=\prod_{P \text{ prime}}  \left(1-P^{-s}\right)^{-p}=\sum_{n=1}^\infty d_{p}(n) n^{-s}.
\end{gather}
for any real number $p$, where  $d_{p}(n)$ denote the generalized divisor function. When $p\geq 0$ it follows that $d_p(n)$ is non negative by the fact that $d_p(n)$ is a multiplicative function and from the fact that  \begin{gather*} (1-P^{-s})^{-p} = \sum_{k=0}^\infty \binom {-p} k (-1)^k  P^{-ks}, \\ \intertext{where} \binom {-p} k (-1)^k  = \prod_{j=1}^k \frac{p+j-1}{j} \geq 0.\end{gather*}
Thus we may apply Lemma \ref{lem1} to obtain the first part of Lemma \ref{lem2}.
To prove the second part we use the following equality for $\Re(s)>1$
\begin{gather*}
 \p{\frac{\zeta(2s)}{\zeta(s)}}^p=\prod_{P \text{ prime}}  \left(1+P^{-s}\right)^{-p}=\sum_{n=1}^\infty \lambda(n) d_{p}(n) n^{-s}.
\end{gather*}
where $\lambda(n)=(-1)^{\nu(n)}$ and where $\nu(n)$ counts the number of prime factors (with multiplicity) of $n$. Since $a(n)$ is the product of a non negative function $d_{p}(n)$ and a completely multiplicative function $\lambda(n)$,   Lemma \ref{lem2} $(ii)$ follows from Lemma \ref{lem1}.
\end{proof}

\subsubsection{A stronger upper bound for $-1<p<1$.}
 We will prove a somewhat stronger result than Theorem \ref{thm2} in the case $-1<p<1$. 
\begin{thm} \label{thm3} Assume that $-1<p<1$ and $\delta>0$. Then
$$ 0.3 \frac{\delta^{1-|p|}}{1-|p|} \leq   \limsup_{T \to \infty} \int_T^{T+\delta} \abs{\zeta(1+it)}^p dt \leq  12  \frac{\delta^{1-|p|}}{1-|p|} +6 \delta(1+\log(1+\delta)). $$
\end{thm}
\begin{proof}

Let $0 \leq q <1$. From Lemma \ref{lem2}, by letting $\sigma \to 1^+$ and using continuity, it follows that
 \begin{gather} \label{ineq1}
      I\p{\frac\delta 2,q}  \leq \limsup_{T \to \infty} \int_T^{T+\delta} \abs{\zeta(1+it)}^q dt  \leq 2I(\delta,q), \\
    \intertext{and} \label{ineq2}
       I\p{\frac\delta 2,q} \leq \limsup_{T \to \infty} \int_T^{T+\delta} \abs{\frac{\zeta(2+2it)}{\zeta(1+it)}}^q dt  \leq 2I(\delta,q), \\ \intertext{where} \notag
      I(\delta,q):= 
       \int_{-\delta}^{\delta}\abs{\zeta(1+it)}^q \p{1-\frac{|t|} \delta}  dt.  
 \end{gather}
  By the Laurent expansion at $s=1$ of the Riemann zeta-function
   and the  bound $|\zeta(1+it)| \leq 1+\log(|t|+1)$ valid for $|t| \geq 1$ we have that 
  \begin{gather*}
  {|t|}^{-q} \leq \abs{\zeta(1+it)}^q \leq  \abs{t}^{-q}+1+\log(|t|+1), 
   \qquad (0 \leq q \leq 1),
\end{gather*}
and it follows by integrating this inequality that 
\begin{gather} \label{ineq3} \frac{2\delta^{1-q}}{(1-q)(2-q)} \leq  I(\delta,q) \leq  \frac{2\delta^{1-q}}{(1-q)(2-q)}+ \delta(1+ \log(\delta+1)).  
\end{gather}
Our lemma follows for $p=q \geq 0$ from \eqref{ineq1} and \eqref{ineq3} with the somewhat stronger lower constant $0.5$ (rather than 0.3) and upper constant $4$ (rather than 12). If $-1<p <0$ we let $q=-p$ and our results follows from  \eqref{ineq2}, \eqref{ineq3} and the inequality
  \begin{gather} \frac 1 3< \abs{\zeta(2+2it)}<\frac 5 3, \qquad (t \in \R).
  \label{line2} \end{gather}
  \end{proof}

\section{Order estimates and Omega estimates}
For the final proof of  Theorem \ref{thm2}  we need  that the $p \,$th moment of the Riemann zeta-function in short intervals on the 1-line is unbounded for $|p| \geq 1$. While this in fact follows from Theorem \ref{thm3} when $p \to 1^-$ we are interested in obtaining more precise  Omega-estimates and Order estimates that answers the question on how fast such integrals might grow.
\subsection{Order estimates}

\begin{thm} \label{thm4} Assuming the Riemann hypothesis then for any $\delta>0$ we have 
\begin{gather*}
 \int_T^{T+\delta} \abs{\zeta(1+it)}^{\pm 1} dt = O(\log \log \log T).
\end{gather*}
\end{thm}
\begin{proof}
  We have that
  \begin{gather*}
    \int_T^{T+\delta}\abs{\zeta(1+it)}^{\pm 1}dt \leq \sup_{t \in [T,T+\delta]} \abs{\zeta(1+it)^{\pm 1}}^{1-p} \int_T^{T+\delta}\abs{\zeta(1+it)}^{\pm p} dt. 
  \end{gather*}
  The result follows from choosing $p=1-1/\log \log \log T$, invoking Littlewood's bound \eqref{A1} on the first part and using Theorem \ref{thm3} on the remaining integral. 
\end{proof}

\begin{thm} \label{thm5} For any $\delta>0$ we have
\begin{gather*}
\int_T^{T+\delta} \abs{\zeta(1+it)}^{\pm 1} dt = O(\log \log T). 
\end{gather*}
\end{thm}
\begin{proof}
  This follows from Theorem \ref{thm3} in the same way as Theorem \ref{thm4}, but by choosing   $p=1-1/\log \log T$ and invoking Vinogradov-Korobov's estimate \eqref{A3} in view of Littlewood's.
\end{proof}

\begin{thm} \label{thm6} Assuming the Riemann hypothesis, then for any $\delta>0$ and $p>1$ the following bound holds true
\begin{gather*}
 \int_T^{T+\delta}  \abs{\zeta(1+it)}^{\pm p} dt = O(\log \log \log T \hskip 1pt (\log \log  T)^{p-1}).  
\end{gather*}
\end{thm} 
\begin{proof}
  We have that
  \begin{gather*}
    \int_T^{T+\delta}\abs{\zeta(1+it)}^pdt \leq \sup_{t \in [T,T+\delta]} \abs{\zeta(1+it)}^{p-1} \int_T^{T+\delta}\abs{\zeta(1+it)} dt. 
  \end{gather*}
 The result follows from Theorem \ref{thm4} and Littlewood's bound \eqref{A1}. 
\end{proof}

\begin{thm} For any $\delta>0$ and $p>1$ the following bound holds true
\begin{gather*}
 \ \int_T^{T+\delta}  \abs{\zeta(1+it)}^{\pm p} dt = O(\log  \log T (\log  T)^{2/3(p-1)}).
\end{gather*}
\end{thm}
\begin{proof}
 This follows from Theorem \ref{thm5} in the same way as Theorem \ref{thm6} follows by Theorem \ref{thm4} by invoking Vinogradov-Korobov's estimate \eqref{A3} in view of Littlewood's bound \eqref{A1}.
\end{proof}

\subsection{Omega estimates}

 We have the following Omega estimates
\begin{thm} \label{thm8} We have for any fixed $\delta>0$ that
\begin{align*}
  &(i) & \qquad \int_T^{T+\delta} \abs{\zeta(1+it)}^{\pm 1} dt&=\Omega(\log \log \log T), \\
  &(ii) & \qquad   \int_T^{T+\delta} \abs{\zeta(1+it)}^{\pm p} dt&=\Omega(( \log \log T)^{p-1}), \qquad (p>1).
  \end{align*}
\end{thm}
We would like to remark that Theorem \ref{thm8} with $p=2$  answers a question of Weber \cite[Problem 6.4]{Weber} in the case $\sigma=1$ in the affirmative\footnote{Ram\=ununas Garunk\v{s}tis remarked that the case $1/2<\sigma<1$ in Weber's problem follows as a direct consequence of the Voronin universality theorem.}. 

 \begin{proof} Again we are going to use a convolution.  Since it is more convenient to have the compact support on the sum side we will consider convolution by the Fourier transform of the triangular function (and higher order convolutions of the triangular function). Define recursively
 \begin{align}
    \theta_1(x)&=\theta(x) \\
  \theta_n(x)&=(\theta_{n- 1}* \theta)(x)= \int_{-\infty}^\infty \theta_{n-1}(t) \theta(x-t) dt
  \end{align}
  It is clear that $0 \leq \theta_n(x) \leq 1$ is a continuous function with support on $[-n,n]$  and by \eqref{ft2} it is clear that its Fourier transform satisfies
    \begin{gather} \label{ere}
     \hat \theta_n(t) = (\hat \theta(t))^n = \p{\frac{\sin \pi t}{ \pi t}}^{2n}.
    \end{gather}
 Consider 
  \begin{align} \label{io2} \zeta_{n,N}^{p}(s)&:=  \int_{-\infty}^\infty \p{\zeta\p{s+i\frac{x} N}}^p \hat \theta_n (x) dx, \qquad \Re(s)>1 \\ \intertext{and}
    \label{io3} Z_{n,N}^{p}(s)&:=  \int_{-\infty}^\infty
    \p{\frac {\zeta\p{2s+2i\frac{x} N}} {\zeta\p{s+i\frac{x} N}} }^{p} \hat \theta_n (x) dx, \qquad \Re(s)>1
    \end{align} 
    where the functions are defined by continuous extension when $\Re(s)=1$ and 
 where $\hat \theta_n(x)$ is given by \eqref{ere}.   In particular $ \zeta_{n,N}^{p}$ is a smoothed version of the $p$'th power of usual Riemann zeta-function. From now on choose $n>p/2$ to be an integer. 
It follows from the convolution \eqref{io2} and the Laurent expansion of the zeta-function at $s=1$\footnote{The integrals  \eqref{io2} and \eqref{io3} should here be interpreted as the limit when $s \to 1^+$.} that 
\begin{gather*} 
   \abs{\zeta_{n,N}^{p}(1+it)} = t^{-p}  \p{1+O(t)+O((Nt)^{p-2n})},    \qquad (N^{-1} \leq t \leq 1).
 \end{gather*}
 Thus, since $p<2n$, by calculus, we have  for fixed $\delta>0$ and $p \geq 1$ that
\begin{gather} \int_0^{\delta} \abs{\zeta_{n,N}^{p}(1+it)} dt \gg \begin{cases} \log N+O(1), & p=1, \\ N^{p-1}, & p>1. \end{cases} \label{ire}
 \end{gather}
By \eqref{io2}, \eqref{io3} and \eqref{thetaref} we have the Dirichlet series expansions
 \begin{gather} 
   \zeta_{n,N}^{p}(s)=\sum_{j=1}^{\lfloor  \exp(nN) \rfloor} \frac {d_{p}(j)} {j^s} \theta_n \p{\frac{\log j} N}, \label{oj1} \\ \intertext{and}
   Z_{n,N}^{p}(s)=\sum_{j=1}^{\lfloor  \exp(nN) \rfloor} \frac {d_{p}(j)\lambda(j)} {j^s} \theta_n \p{\frac{\log j} N}, \label{oj2}
   \end{gather}
By \eqref{ere}, \eqref{io2}, \eqref{io3} and the triangle inequality it follows for $T \geq T_0$ sufficiently large and $N \geq 1$ that
\begin{gather} \label{aj}
  0.1 \int_{T}^{T+\delta} \abs{\zeta_{n,N}^{p}(1+it)} dt \leq  \max_{T/2 \leq X \leq 2T} \int_{X}^{X+\delta} \abs{\zeta(1+it)}^{p} dt, \\ \intertext{and} \label{ajajaj}
  0.1 \int_{T}^{T+\delta} \abs{Z_{n,N}^{p}(1+it)} dt \leq  \max_{T/2 \leq X \leq 2T} \int_{X}^{X+\delta} \abs{\frac{\zeta(2+2it)} {\zeta(1+it)}}^{p} dt.
\end{gather}
Thus it is sufficient to bound the left hand side of \eqref{aj}. By Dirichlet's approximation theorem there exists for each $N > 0$ some \begin{gather} \label{TNineq} 0 \leq T_N \leq N^{4 \pi(\exp(nN))}, \end{gather} where $\pi(\exp(nN))$ here denote the number of primes less than $\exp(nN)$,  
such that
\begin{gather} \label{io}
 \operatorname{dist}\p{\frac{T_N \log P}{2 \pi}, \mathbb{Z}}<N^{-4}, \qquad (P \text{ prime}, 2 \leq P \leq \exp(nN)).
\end{gather}
It follows from \eqref{io} that\footnote{The worst case is when $n$ is a power of 2}
\begin{gather}
  \abs{j^{i{T_N}}-1}<N^{-2}, \qquad (1 \leq j \leq \exp(nN))
\end{gather}
For such a $T_N$ it follows from \eqref{oj1} that 
\begin{gather} \label{aj99a}
  \abs{\zeta_{n,N}^p(1+iT_N+it)-\zeta_{n,N}^p(1+it)} \leq N^{-1}, \qquad (t \in \R)
\end{gather}
 By noticing that it follows from $\eqref{TNineq}$ that $\log \log T_N \ll N$ and   the inequalities \eqref{ire}, \eqref{aj} and \eqref{aj99a}  it follows that
\begin{gather*} 
     \max_{T_N/2 \leq X \leq 2T_N} \int_{X}^{X+\delta} \abs{\zeta(1+it)}^{p} dt  \gg \begin{cases} \log \log \log T_N, & p=1, \\  (\log \log T_N)^{p-1}, & p>1, \end{cases}
\end{gather*}
which gives our result for $p \geq 1$.
For negative moments $p \leq -1$  we need some corresponding result for $Z_{n,N}^p(s)$. It is not sufficient to use the Dirichlet approximation theorem directly and  we need some effective variant of the Kronecker approximation theorem. Bohr and Landau \cite{BohrLandau1,BohrLandau3}\footnote{see  also the discussion in \cite[p.182]{Steuding}} proved that
\begin{gather}
  \abs{P^{iT}+1}<\frac 1 M, \qquad (1 \leq P \leq M, \, P \text{ prime}),
\end{gather}
holds for some $0<T<\exp(M^6)$. By using $M=\lfloor \exp(nN) \rfloor$ and $T_N=T$ it follows from \eqref{oj2} that
\begin{gather} \label{aj99b}
  \abs{Z_{n,N}^p(1+iT_N+it)-\zeta_{n,N}^{-p}(1+it)} \leq 1, \qquad (t \in \R)
\end{gather}
holds for some $0 \leq  T_N \leq \exp(\exp(6nN))$. Thus by combining \eqref{ire}, \eqref{ajajaj}, \eqref{TNineq}, \eqref{line2} and \eqref{aj99b} 
it follows that
\begin{gather*} 
\max_{T_N/2 \leq X \leq 2T_N} \int_{X}^{X+\delta} \abs{\zeta(1+it)}^{-p} dt \gg \begin{cases} \log \log \log T_N, & p=1, \\ (\log \log T_N)^{p-1}, & p>1. \end{cases} 
 \end{gather*}
\end{proof}

\begin{ack}
  I would like to thank Ram\=ununas Garunk\v{s}tis who in April 2020 asked about some related problems  that gave me motivation to finish this manuscript\footnote{Most of the results of this paper were developed in 2011-2013 in an unfinished manuscript.}.
\end{ack}

\section{Proof of Theorem \ref{thm2}}

Theorem \ref{thm2} follows from Theorem \ref{thm3} and Theorem \ref{thm8}. \qed

\def\cprime{$'$} \def\polhk#1{\setbox0=\hbox{#1}{\ooalign{\hidewidth
  \lower1.5ex\hbox{`}\hidewidth\crcr\unhbox0}}}

\bibliographystyle{plain}

\def\cprime{$'$} \def\polhk#1{\setbox0=\hbox{#1}{\ooalign{\hidewidth
  \lower1.5ex\hbox{`}\hidewidth\crcr\unhbox0}}}

\end{document}